\documentclass[reqno,12pt,a4paper]{amsart}

\usepackage{amsmath,amsthm,amssymb,amsfonts,dsfont,ifpdf,mathtools,verbatim}
\usepackage{enumitem}
\usepackage{graphicx}
\usepackage[usenames,dvipsnames]{color}
\usepackage[all]{xy}

\usepackage[colorlinks=true, linkcolor=blue, urlcolor=blue, citecolor=ForestGreen]{hyperref}
\numberwithin{equation}{section}

\newtheorem {thm}    {Theorem}[section]
\newtheorem {lem}      [thm]    {Lemma}

\newtheorem {cor}  [thm]    {Corollary}
\newtheorem {prop}[thm]    {Proposition}
\newtheorem* {prop*} {Proposition}

\newtheorem*{claim*}   {Claim}

\newtheorem*{conj*} {Conjecture}

\theoremstyle{definition}
\newtheorem {defn} [thm]    {Definition}

\newtheorem {rmk}    [thm]    {Remark}
\newtheorem*{rmk*}  {Remark}

\newtheorem*{qst*} {Question}

\newtheorem* {problem*}{Problem}

\newcounter{AbcT}

\numberwithin{equation}{section}

\newcommand {\supl}   {\sup\limits}
\newcommand {\infl}   {\inf\limits}

\newcommand {\limsupl}   {\limsup\limits}

\newcommand {\E} {{\mathbb E}}

\newcommand {\N} {{\mathbb N}}

\renewcommand {\P}  {{\mathbb P}}
\newcommand {\R} {{\mathbb R}}

\newcommand {\cP} {{\mathcal P}}

\DeclareMathOperator{\supp}{supp}

\newcommand{\eps}{\varepsilon}

\newcommand {\IGNORE}[1]  {}

\newcommand {\La} {{\Lambda}}
\newcommand{\Ga}{\Gamma}

\newcommand{\lb}{\liminf\limits_{n\to\infty}\frac{1}{n}\log{\mu_n}}
\newcommand{\ub}{\limsup\limits_{n\to\infty}\frac{1}{n}\log{\mu_n}}

\begin{document}

	\title[Large deviations for random walks on free products]{Large deviations for irreducible random walks on relatively hyperbolic groups}
	\author[E.~Corso]{Emilio Corso}
	\address[E. C.]{ETH Z\"urich, R\"amistrasse 101
		CH-8092 Z\"urich
		Switzerland}
	\email{emilio.corso@math.ethz.ch}
	\date{\today}
	
	
	\begin{abstract}
		We show existence of the weak large deviation principle, with a convex rate function, for the renormalized distance from the starting point of irreducible random walks on relatively hyperbolic groups. Under the assumption of finiteness of exponential moments, the full large deviation principle holds, and the rate function governing it can be expressed as the Fenchel-Legendre transform of the limiting logarithmic moment generating function of the sequence of renormalized distances. 
	\end{abstract}
	\maketitle
	
	\tableofcontents

	\section{Introduction and main result}
	
	The purpose of this article is to extend the method and the results of~\cite{Corso}, concerning the existence of the large deviation principle for the renormalized distance function of random walks on free products of finitely generated groups, to irreducible random walks on relatively hyperbolic groups.
	
	The notion of a relatively hyperbolic group was first introduced by Gromov in~\cite{Gromov-hyperbolicgroups}, in order to encompass various classes of groups of algebraic and geometric nature, such as fundamental groups of non-compact finite-volume Riemannian manifolds of pinched negative sectional curvature, geometrically finite Kleinian groups, small cancellation quotients of free products, hyperbolic groups and so forth. Gromov's intuition was later elaborated by Bowditch in~\cite{Bowditch} and, in an alternative form, by Farb in~\cite{Farb}. The  different definitions of Bowditch and Farb were shown to be inequivalent by Szczepa\'{n}ski (\cite{Szczepanski}); however, it turns out that Farb's definition of relative hyperbolicity, in conjunction with the Bounded Coset Penetration property (cf.~Remark~\ref{equivalence}) he himself introduced, amounts precisely to Bowditch's definition. The reader is referred to Section~\ref{relhypgroups} for the necessary background on relatively hyperbolic groups.

	In this manuscript, we shall be concerned with asymptotic properties of random walks on relatively hyperbolic groups. For a brief history summarizing the state of the art regarding asymptotic limit theorems for random walks on discrete groups, we refer to the introduction to~\cite{Corso} and the references therein.
	
	\smallskip
	Let $G$ be a finitely generated group, hyperbolic relative to a collection $\Omega$ of peripheral subgroups. Denote by $e$ its identity element. We shall always assume that $G$ is non-elementary, that is, it does not contain any cyclic subgroup of finite index.
	Choose a finite set $S$ generating the group $G$, and select a finite set of representatives $\{H_1,\dots,H_r \}$ for the conjugacy classes inside $\Omega$. Let $\ell\colon G \to [0,\infty)$ be the word length induced on $G$ by the (typically infinite) generating set $S\cup\bigcup_{1\leq i\leq r}H_i$, that is, $\ell$ is defined by
	\begin{equation*}
	\ell(g)\coloneqq \inf\biggl\{n\in \N_{\geq 1}: \text{ there exist }x_1,\dots,x_n\in S\cup S^{-1}\cup\bigcup_{1\leq i\leq r}H_i \text{ s.t. }g=x_1\cdots x_n  \biggr\}\;,
	\end{equation*}
	with the understanding that $\ell(e)=0$, and where $S^{-1}=\{s^{-1}:s\in S \}$.
	
	\smallskip
	Suppose given a probability measure $\mu$ on $G$ viewed as a discrete group; in other words, $\mu\colon G \to [0,1]$ is a function satisfying $\sum_{g\in G}\mu(g)=1$. Then $\mu$ defines a random walk on $G$ obtained by successively multiplying elements of $G$ drawn randomly and independently according to $\mu$. More precisely, let $(X_{i})_{i\geq 1}$ be a sequence of independent $G$-valued random variables distributed according to $\mu$. Define a stochastic process $(Y_{n})_{n\in \N}$ by setting $Y_0\coloneqq e$, $Y_{n}=X_1\cdots X_n$ for every integer $n\geq 1$. Then $(Y_{n})_{n\in \N}$ is called a right random walk on $G$, issued from the identity, with increments distributed according to the law $\mu$. Alternatively, $(Y_{n})_{n\in \N}$ can be defined as the Markov chain on $G$ issued from the identity with transition matrix $Q=(q(x,y))_{x,y \in G}$ given by $q(x,y)=\mu(x^{-1}y)$ for every $x,y\in G$. Henceforth, we shall denote by $\P$ the law of the process $(Y_n)_{n\in \N}$ on the Borel space $(G^{\N},\cP(G)^{\otimes \N})$, where $\cP(G)$ is the power set of $G$. 
	
	\medskip
	We are interested in large deviation estimates for the process $\bigl(\frac{1}{n}\ell(Y_n)\bigr)_{n\geq 1}$ of renormalized lengths of the random walk. By virtue of Kingman's sub-additive ergodic theorem (\cite{Kingman}), it is well-known that, whenever $\mu$ has finite first moment\footnote{We say that $\mu$ has a finite first moment if $\int_{G}\ell(g)\text{d}\mu(g)<\infty$, and has a finite second moment if $\int_{G}\ell(g)^{2}\text{d}\mu(g)<\infty$.}, then $\frac{1}{n}\ell(Y_n)$ converges almost surely to the \emph{escape rate} $\lambda\in \R_{\geq 0}$ of the random walk, defined as $\lambda=\lim_n\frac{1}{n}\E[\ell(Y_n)]$ (see also~\cite{Guivarch}). Both the existence of the latter limit and the almost-sure convergence hinge upon the sub-additive nature of the length function $\ell$, that is, on the fact that $\ell(gh)\leq \ell(g)+\ell(h)$ for any $g,h \in G$.  
	
	It is worth mentioning that a central limit theorem has also been established in the context of relative, and more generally acylindrically hyperbolic groups by Mathieu and Sisto (\cite{Mathieu-Sisto}),
	whenever $\mu$ has a finite second moment. For local limit theorems, we refer instead to the work of Dussaule in~\cite{Dussaule,Dussaule-two}.
	
	As almost-sure convergence implies convergence in probability, it follows from the aforementioned analogue of the strong law of large numbers that $\P(|\ell(Y_n)-n\lambda|\geq \delta n)$ tends to zero as $n$ tends to infinity, for any $\delta>0$. The large deviation principle (cf.~Definition~\ref{LDP}) quantifies the exponential rate of decay of the probability of such rare events.
	
\smallskip
 Let us call a probability measure $\mu$ on $G$ \emph{admissible} if the random walk $(Y_{n})_{n\geq 0}$ it generates on $G$ is irreducible: for any $g\in G$ there exists an integer $n\geq 1$ such that $\P(Y_{n}=g)>0$. Equivalently, $\mu$ is admissible if its topological support $\supp{\mu}=\{g\in G:\mu(g)>0 \}$ generates $G$ as a semigroup.
 
 \smallskip
 The following is the main result of this article. 
	
	\begin{thm}
	\label{main}
		Let $G$ be a finitely generated group, hyperbolic relative to a collection\linebreak $\Omega$ of peripheral subgroups. Assume $G$ is non-elementary. Let $S\subset G$ be a finite set generating $G$, $\{H_1,\dots,H_r\}$ a complete set of representatives of the conjugacy classes in $\Omega$, $\ell$ the word length determined by the union $S\cup \bigl(\bigcup_{1\leq i\leq r}H_i\bigr)$. Suppose $\mu$ is an admissible probability measure on $G$, and let $(Y_{n})_{n\geq 0}$ be a right random walk on $G$ issued from the identity with increments distributed according to $\mu$. 
		Then the sequence of random variables 
			$\bigl(\frac{1}{n}\ell(Y_n)\bigr)_{n\geq 1}$
			satisfies the weak large deviation principle with a convex rate function $I\colon \R_{\geq 0}\to [0,\infty]$.
		\end{thm}
		
		Observe that both the validity of the weak large deviation principle and the convexity of the rate function hold irrespective of any moment assumption on $\mu$. The hypothesis placed on its support is only of algebraic type. On the other hand, standard arguments from the abstract theory of large deviations enable us to get a full large deviation principle, as well as a more precise description of the rate function, provided that $\mu$ has finite exponential moments. Recall that $\mu$ is said to have a finite exponential moment if there exists $\tau>0$ such that $\int_{G}\exp{(\tau\ell(g))}\;\text{d}\mu(g)$ is finite, and has a finite moment generating function if $\int_{G}\exp{(\tau\ell(g))}\;\text{d}\mu(g)$ is finite for every $\tau>0$.
		
		\smallskip
		 As a consequence of Theorem~\ref{main}, we thus get:
		
		\begin{cor}
		\label{expmoments}	
			Under the assumptions of Theorem~\ref{main}, the following hold:
			\begin{enumerate}
			\item if $\mu$ has a finite exponential moment, then $I$ is topologically a proper function and the sequence $\bigl(\frac{1}{n}\ell(Y_n)\bigr)_{n\geq 1}$ satisfies the full large deviation principle with rate function $I$;
			\item If $\mu$ has a finite moment generating function, then $I$ is the Fenchel-Legendre transform of the limiting logarithmic moment generating function of the sequence $\bigl(\frac{1}{n}\ell(Y_{n})\bigr)_{n\geq 1}$.
		\end{enumerate} 
	\end{cor}

Section~\ref{sectionLDP} contains an explanation of the terminology adopted in the statement.

\smallskip
The deduction of Corollary~\ref{expmoments} from~Theorem~\ref{main} does not rely on any distinctive feature of the underlying group, and has been carried out in~\cite[Prop.~4.4]{Corso} and~\cite[Sec.~5.2]{Corso}; we don't reproduce it here. For further properies of the rate function, we also refer to~\cite[Sec.~5.1]{Corso}. 

\begin{rmk}
\label{moregeneral}
	\begin{enumerate}
	\item The validity of Theorem~\ref{main} merely hinges upon the geometric property of the word distance phrased in Proposition~\ref{boundeddistortion}, which is shared by several other classes of geodesic metric spaces on which a non-elementary discrete group acts by isometries with contracting elements\footnote{The author wishes to thank W.~Yang for drawing his attention to this fact.} (cf.~\cite[Lem.~2.14]{Yang}). As a consequence, our main result applies to irreducible random walks in such more general circumstances as well. By way of example, we mention the following (see the introduction to~\cite{Yang} and the references therein):
	\begin{enumerate}
		\item groups acting properly and cocompactly on Gromov-hyperbolic metric spaces;
		\item groups acting properly and cocompactly on CAT$(0)$ spaces with rank-one elements;
		\item relatively hyperbolic groups acting on their Cayley graph with respect to a finite generating set;
		\item mapping class groups of closed orientable surfaces of genus at least two acting on the associated Teichm\"{u}ller space endowed with the Teichm\"{u}ller metric.
	\end{enumerate}
	A large deviation principle for non-elementary random walks on Gromov-hyperbolic metric spaces has been previously established by Boulanger, Mathieu, Sert and Sisto in~\cite{Boulanger-Mathieu-Sert-Sisto}. Of the four classes of examples mentioned above, (b),(c) and (d) are not covered by their result~\cite[Thm.~1.2]{Boulanger-Mathieu-Sert-Sisto}; to the best of the author's knowledge, the large deviation principle for them appears to be new. 
	\item A corollary of~\cite[Thm.~1.2]{Boulanger-Mathieu-Sert-Sisto} is a large deviation principle for random walks on the \emph{coned-off} Cayley graph of a relatively hyperbolic group (cf.~the appendix to~\cite{Osin}). As shown in~\cite[Lem.~6.8]{Osin}, the coned-off Cayley graph is only quasi-isometric to the Cayley graph $\text{Cay}\bigl(G,S\cup \bigcup_{1\leq i\leq r}H_i\bigr)$ we are considering; as a consequence, it is not possible to deduce Theorem~\ref{main} directly from~\cite[Thm.~1.2]{Boulanger-Mathieu-Sert-Sisto}, nor viceversa. 
	
	For Gromov-hyperbolic groups, that is, when the collection of peripheral subgroups reduces to the identity subgroup, the two Cayley graphs coincide. In this case, our approach provides a streamlined argument for the large deviation principle, in comparison to the proof of~\cite[Thm.~1.2]{Boulanger-Mathieu-Sert-Sisto}; on the other hand, the latter yields finer information on the rate function, notably uniqueness of the zero.
	\end{enumerate}
\end{rmk}

\subsection{An outline of the proof} The proof of Theorem~\ref{main} proceeds along the same lines as the proof of~\cite[Thm.~1.4]{Corso}, drawing inspiration from Lanford's indirect argument (\cite{Lanford}) for Cramer's theorem on the large deviation principle for empirical means $\frac{1}{n}\sum_{i=1}^{n}X_i$ of i.i.d.~real random variables. As explained in~\cite[Sec.~1.1]{Corso}, the latter relies fundamentally on the additivity property of the partial sums $\sum_{i=1}^{n}X_i$. In our case, the word length $\ell$ is only subadditive, that is, $\ell(X_1\cdots X_{n+m})\leq\ell(X_1\cdots X_n)+\ell(X_{n+1}\cdots X_{n+m})$ for every $n,m\geq 1$. It is precisely this lack of additivity that constitutes the main hurdle when attempting to adapt Lanford's proof to renormalized lengths of random walks on finitely generated groups.

Nevertheless, random walks on relatively hyperbolic groups possess the feature that successive steps tend to be almost aligned (that is, the length function is almost additive) with sufficiently high probability on an exponential scale, in a way which is formally expressed by Lemma~\ref{selection}; see also the deviation inequalities proven in~\cite{Mathieu-Sisto}, which point to the same phenomenon. We derive this property from a well-known geometric characteristic of Cayley graphs of relatively hyperbolic groups, phrased in Proposition~\ref{boundeddistortion}. Existence of the large deviation principle, as well as convexity of the rate function, is then achieved through classical arguments from the general theory of large deviations, notably Proposition~\ref{criterion}.

\subsection*{Acknowledgments}
The author is thankful to \c{C}agri Sert for pointing out the crucial geometric ingredient (Proposition~\ref{boundeddistortion}) of the proof of Theorem~\ref{main}, and acknowledges his indebtedness to Wenyuan Yang for numerous valuable observations.

\section{Preliminaries on relatively hyperbolic groups and large deviations}

\subsection{Relatively hyperbolic groups}
\label{relhypgroups}

We begin this preliminary section by reviewing the notion of a relatively hyperbolic group. Several equivalent definitions of a relatively hyperbolic group can be given (see~\cite{Bowditch,Farb,Osin}), of which the most convenient for our purposes is combinatorial in nature, and was introduced by Farb in~\cite{Farb}. We follow Osin's treatment, as presented  in~\cite{Dussaule}.

Let $G$ be a finitely generated group and $\Omega$ a collection of subgroups of $G$. Assume that $\Omega$ is closed under conjugation by elements of $G$ and admits a finite number of conjugacy classes. Choose a finite set $\Omega_0=\{H_1,\dots,H_r \}\subset \Omega$ of representatives of such conjugacy classes. 

Let $S\subset G$ be a finite  set generating $G$, and denote by $\text{Cay}(G,S)$ the associated Cayley graph (cf.~\cite[Chap.~IV]{deLaHarpe}). We shall also consider the Cayley graph $\text{Cay}\bigl(G,S\cup\bigcup_{1\leq i\leq r}H_i\bigr)$ with respect to the generating set obtained by adding to $S$ all elements of the subgroups $H_i$; this is sometimes referred to in the literature as the \emph{relative} Cayley graph of $G$. 

It is well-known (see, for instance, \cite{Bridson-Haefliger}) that any undirected graph can be endowed with the structure of a complete geodesic metric space, for which all edges in the graph are isometrically isomorphic to the unit interval $[0,1]$, endowed with the Euclidean distance. We shall denote by $\Ga$ the metric space associated in this way to $\text{Cay}\bigl(G,S\cup\bigcup_{1\leq i\leq r }H_i\bigr)$.

Henceforth, we denote by $d$ the graph distance in $\text{Cay}\bigl(G,S\cup\bigcup_{1\leq i\leq r }H_i\bigr)$; this is the distance\linebreak we shall always consider for the purposes of establishing the large deviation principle (cf.~ Proposition~\ref{boundeddistortion}). Equipped with this distance, the Cayley graph embeds isometrically into $\Ga$. The graph distance in $\text{Cay}(G,S)$ will be indicated with $d'$ instead. 

\smallskip
Recall that a geodesic metric space $(X,d)$ is called hyperbolic if there exists $\delta>0$ such that, for any geodesic triangle in $X$, each edge is contained in the closed $\delta$-neighborhood of the union of the other two edges (cf.~\cite{Gromov-hyperbolicgroups,Bridson-Haefliger,deLaHarpe}).

We say that the group $G$ is \emph{weakly relatively hyperbolic} with respect to the collection of subgroups $\Omega$ if $\Ga$ is a Gromov-hyperbolic metric space. The definition does not depend neither on the choice of $\Omega_0$ nor of $S$. 

\smallskip
The definition of relative hyperbolicity builds additionally upon the Bounded Coset Penetration property, which we now address for the sake of completeness, referring to~\cite{Osin} for its geometric significance\footnote{We just remark here that mapping class groups are weakly relatively hyperbolic, as shown by Masur and Minsky in~\cite{Masur-Minsky} through the study of their action on the curve complex of the associated surface, but do not satisfy the Bounded Coset Penetration property. }. In what follows, a \emph{path} in $\Ga$ is a finite sequence of adjacent vertices\linebreak in $\text{Cay}\bigl(G,S\cup\bigcup_{1\leq i\leq r}H_i\bigr)$. Given real numbers $\lambda,c>0$, a \emph{relative $(\lambda,c)$-quasi geodesic path} is a path $\alpha=(\gamma_1,\dots,\gamma_n)$ which is a $(\lambda,c)$-quasi geodesic, that is, it satisfies
\begin{equation*}
\lambda^{-1}|k-l|-c\leq d(\gamma_k,\gamma_l)\leq \lambda|k-l|+c \quad \text{ for any }k,l\in \{1,\dots,n \}.
\end{equation*} 
We say that a path $\alpha$ \emph{enters} the left coset $gH_i$ if it has a vertex which is an element of $gH_i$ and which is followed by an edge corresponding to an element of $H_i$. For such a path, consider a maximal subpath consisting of vertices in $gH_i$ and whose edges are labelled with elements of $H_i$. We call such a subpath a $H_i$-\emph{component}. The entering (resp.~exit) point of the $H_i$-component is the first (resp.~last) vertex of such a subpath; we say that the subpath \emph{leaves} $gH_i$ at the exit point. 
Moreover, given a non-negative integer $r$, the path is said to \emph{travel more} than $r$ in $gH_i$ if the distance in $\text{Cay}(G,S)$ between the entering and the exit point is larger than $r$.

Lastly, a path is \emph{without backtracking} if it never goes back to a coset $gH_i$ once it has left it.
 
\smallskip
By definition, the pair $(G,\Omega)$ satisfies the Bounded Coset Penetration property if, for all $\lambda,c>0$, there exists a constant $C_{\lambda,c}$ such that, for every pair $(\alpha_1,\alpha_2)$ of relative $(\lambda,c)$-quasi geodesic paths without backtracking and with common end points, the following hold:

\begin{itemize}
	\item[--] if $\alpha_1$ travels more than $C_{\lambda,c}$ in a coset $gH_i$, then $\alpha_2$ enters $gH_i$;
	\item[--] if $\alpha_1$ and $\alpha_2$ enter the same coset $gH_i$, the two entering points and the two end points are $C_{\lambda,c}$-close to each other in $\text{Cay}(G,S)$.  
\end{itemize} 

We can now finally give:

\begin{defn}
	We say that the group $G$ is relatively hyperbolic with respect to $\Omega$ if it is weakly relatively hyperbolic with respect to $\Omega$ and the pair $(G,\Omega)$ satisfies the Bounded Coset Penetration property.
\end{defn}

Following the literature, we refer to the elements of the collection $\Omega$ as peripheral subgroups.

\begin{rmk}
\label{equivalence}
	The definition we just gave is equivalent to the one given by Bowditch in~\cite{Bowditch}, characterizing relatively hyperbolic groups in terms of properly discontinuous isometric actions on proper Gromov-hyperbolic metric spaces. For the equivalence, see~\cite{Bowditch} or~\cite{Osin}.

\end{rmk}

For the reader's convenience, we now present a short list of classical examples of relatively hyperbolic groups, taken from the introduction to~\cite{Osin} (to which we refer for further examples).

\begin{enumerate}
	\item Let $M$ be a complete, connected, finite-volume non-compact Riemannian manifold with pinched negative sectional curvature $K$, namely satisfying $-b^{2}\leq K\leq -a^{2}$ for some real numbers $a,b>0$. Then the fundamental group $\pi_1(M)$ of $M$ is relatively hyperbolic with respect to the collection of cusp subgroups (cf.~\cite{Eberlein,Farb}). Examples include non-uniform lattices in real simple Lie groups of real rank one.
	\item Let $G$ be a $C'(1/6)$-small cancellation quotient (cf.~\cite[Chap.~V]{Lyndon-Schupp}) of the free product of groups $G_1,\dots,G_r$. Then $G$ is hyperbolic relative to the collection of the canonical images of the subgroups $G_i$ in $G$; for the reason, we refer to~\cite{Osin}.
	\item Let $G$ be a Gromov-hyperbolic group, and assume $H_1,\dots,H_r$ are quasi-convex subgroups of $G$. Suppose that the cardinality $|gH_ig^{-1}\cap H_j|$ is finite whenever $i\neq j$ or $g\notin H_i$. Then $G$ is hyperbolic relative to the collection $H_1,\dots,H_r$: for this, see~\cite{Farb} and~\cite[Thm.~7.11]{Bowditch}. In particular, every Gromov-hyperbolic group is relatively hyperbolic with respect to the trivial subgroup.
\end{enumerate}

\subsection{The large deviation principle}
\label{sectionLDP}
 This subsection sets forth the necessary terminology of the theory of large deviations, which we adopt throughout the manuscript. We also recall a general criterion to establish existence of the large deviation principle, which will be put to good use in Section~\ref{proofmain} for the proof of Theorem~\ref{main}. We refer the reader to~\cite{De-Ze} for a thorough treatment of the subject.

\smallskip
Let $X$ be a Hausdorff regular topological space, $I\colon X\to[0,\infty]$ a lower semicontinuous function, $(\mu_n)_{n\geq 1}$ a sequence of Borel probability measures on $X$.
\begin{defn}
	\label{LDP}
	We say that the sequence $(\mu_n)_{n\geq 1}$ satisfies the full large deviation principle (LDP for short) with rate function $I$ if, for any Borel measurable set $\La\subset X$, 
	\begin{equation}
	\label{ldp}
	-\infl_{x\in \La^{\circ}}I(x)\leq \lb(\La)\leq \ub(\La)\leq-\infl_{x \in \overline{\La}}I(x)\;,
	\end{equation}
	where $\La^{\circ}$ and $\overline{\La}$ denote the interior and the closure of $\La$, respectively.
	
	If~\eqref{ldp} only holds for any measurable $\Lambda$ with compact closure, the sequence is said to satisfy the weak LDP with rate function $I$. 
\end{defn}

The following general principle allows to establish existence of the weak LDP in an indirect way, that is, without knowing the rate function in advance.  

\begin{prop}[{\cite[Thm.~4.1.11]{De-Ze}}]
	\label{criterion}
	Let $(\mu_n)_{n\geq 1}$ be a sequence of Borel probability measures on $X$. The function $I\colon X \to[0,\infty]$ defined as
	\begin{equation}
	\label{suplower}
	I(x)\coloneqq\supl_{x\in V\emph{open}}-\lb(V)\;, \quad x \in X.
	\end{equation}
is lower semicontinuous. Moreover, if 
	\begin{equation}
	\label{supupper}
	I(x)=\supl_{x\in V\emph{open}}-\ub(V)
	\end{equation}
	for all $x\in X$, then the sequence $(\mu_n)_{n\geq 1}$ satisfies the weak LDP with rate function $I$.
\end{prop}

Once a LDP has been shown via Proposition~\ref{criterion}, it is nevertheless desirable to obtain a more explicit expression for the rate function. This can be achieved by means of the classical theory of convex conjugate functions. Suppose $X$ is a locally convex, Hausdorff topological vector space over $\R$, and denote by $X^{*}$ its topological dual. Assume that a sequence $(\mu_n)_{n\geq 1}$ of Borel probability measures on $X$ satisfies the LDP with a proper, convex rate function $I$. Define the logarithmic moment generating function of the measure $\mu_n$, for each integer $n\geq 1$, as the function $\La_{n}\colon X^{*} \to (-\infty,\infty]$ given by
\begin{equation*}
\La_{n}(\varphi)=\log{\int_{X}e^{\langle \varphi,x\rangle }}d\mu_n(x) \quad\text{for all }\varphi\in X^{*},
\end{equation*}
where $\langle \cdot,\cdot \rangle$ denotes the standard dual pairing between $X^{*}$ and $X$. The limiting logarithmic moment generating function of the sequence $(\mu_n)_{n\geq 1}$ is then defined as
\begin{equation*} 
\La(\varphi)=\limsupl_{n\to\infty}\frac{1}{n}\;\La_n(n\varphi) \in (-\infty,\infty]\quad\text{for all }\varphi \in X^{*}.
\end{equation*}

Also, given a function $f\colon X \to (-\infty,\infty]$, not identically infinite, we define its Fenchel-Legendre transform $f^{*}\colon X^{*}\to (-\infty,\infty]$ as 
\begin{equation*}
f^{*}(\varphi)=\supl_{x\in X}\{\langle \varphi,x\rangle -f(x) \} \quad\text{for all }\varphi\in X^{*}. 
\end{equation*}
If $g\colon X^{*}\to (-\infty,\infty]$ is a function defined on the dual space, its Fenchel-Legendre transform $g^{*}$ will be considered as a function defined just on $X$, rather than on the entire bidual $X^{**}$.  

Varadhan's integral lemma (\cite[Thm.~4.3.1]{De-Ze}), coupled with Fenchel-Moreau's duality theorem (\cite[Thm.~1.11]{Brezis}), yields the following characterization of the rate function:

\begin{prop}[{\cite[Thm.~4.5.10]{De-Ze}}]
	\label{momentgen}
	Let $(\mu_n)_{n\geq 1}$ be a sequence of Borel probability measures on a locally convex, Hausdorff topological vector space $X$. Assume that:
	\begin{enumerate}
		\item  the limiting logarithmic moment generating function $\La \colon X^{*}\to (-\infty,\infty]$ of the sequence $(\mu_n)_{n\geq 1}$ is finite for every $\varphi\in X^{*}$;		 
		\item  the sequence $(\mu_{n})_{n\geq 1}$ satisfies the full LDP with a proper, convex rate function $I$. 
	\end{enumerate} 
	Then the rate function $I$ is the Fenchel-Legendre transform of $\La$, that is,
	\begin{equation*}
	I(x)=\supl_{\varphi \in X^{*}}\{\langle \varphi,x\rangle -\La(\varphi)  \} \text{ for every }x\in X.
	\end{equation*}
\end{prop}

Proposition~\ref{momentgen} justifies the relevance of knowing \emph{a priori} the existence of the full LDP with a proper, convex rate function, as established in our context in Theorem~\ref{main}. 

\section{Proof of the main theorem}
\label{proofmain}

The fundamental geometric ingredient entering the proof of Theorem~\ref{main} is the following geometric observation, initially due to Gou\"{e}zel (\cite[Lem.~2.4]{Gouezel}) for Gromov-hyperbolic groups and later vastly generalized by Yang (\cite[Lem.~2.14]{Yang}) to groups acting isometrically with contracting elements on geodesic metric spaces. We recall that $d$ denotes the graph distance on $\text{Cay}\bigl(G,S\cup\bigcup_{1\leq i\leq r}H_i\bigr)$, so that in particular $d(e,x)=\ell(x)$ for any $x\in G$, while $d'$ is the graph distance on $\text{Cay}(G,S)$. For any real $R>0$, let $B^{G}_{d'}(e,R)$ be the closed $d'$-ball of radius $R$ centered at the identity. 

\begin{prop}[{\cite[Lem.~5.3]{Dussaule}}]
\label{boundeddistortion}
	Let $G$ be a non-elementary relatively hyperbolic group. There exist constants $c,C>0$ such that, for any $x,y\in G$, there exists $\sigma\in B^G_{d'}(e,C)$ such that $d(e,x\sigma y)\geq d(e,x)+d(e,y)-c$.
\end{prop}

In other words, upon perturbing the product of two elements by a bounded amount, there is a controlled loss of additivity in the length function.

The cardinality of the finite set $B^{G}_{d'}(e,C)$ will be indicated with $|B^{G}_{d'}(e,C)|$ in the sequel.

\medskip
In what follows, we shall denote by $B(y,\eps)$ the open interval $(y-\eps,y+\eps)\subset \R$, for any $y\in \R$ and any $\eps>0$. Furthermore, for any positive integer $k$, we let
\begin{equation*}
k B(y,\eps)=\{k z: z \in B(y,\eps)\}.
\end{equation*}

Hereinafter, $\mu_n$ indicates the law of the random variable $\frac{1}{n}\ell(Y_n)$, for every integer $n\geq 1$.

The lemma that follows is instrumental in proving a lower bound for the asymptotic exponential rate of decay of certain probabilities, once a uniform bound on a non-lacunary sequence of times is known. 

\begin{lem}[{\cite[Lem.~4.1]{Corso}}]
	\label{lowerbound}
	Suppose that there exist $a>0,\gamma \in \R$, a strictly increasing sequence $(n_k)_{k\geq 1}$ of positive integers with $\lim_{k\to\infty}n_{k+1}/n_k=1$, such that 
	\begin{equation}
	\label{multiplicative}
	\mu_{n_k}(B(x,a))\geq e^{n_k\gamma} \text{ for all }k\geq 1.
	\end{equation}
	Then, for all $b>a$, 
	\begin{equation*}
	\lb(B(x,b))\geq \gamma\;.
	\end{equation*}
\end{lem}
For a proof of the lemma, we refer to~\cite{Corso}.

\medskip
We shall now show how Proposition~\ref{boundeddistortion} enables us to confine the random walk to certain subsets on which the length function $\ell$ is almost additive.
If $A$ is a subset of $G$ and $k\geq 1$ is an integer, we denote by $A^{k}$ the set of all $k$-tuples $(a_1,\dots,a_k)$ of elements $a_1,\dots,a_k$ drawn from the set $A$. If $\nu$ is a probability measure on a subset $A$ of $G$, $\nu^{k}$ designates the $k$-fold product of $\nu$ with itself, as a probability measure on the product space $A^{k}$.

\smallskip
The following lemma provides an adequate replacement for~\cite[Lem.~4.2]{Corso}.

\begin{lem}
\label{selection}
	Let $F\subset G$ be a subset, $\nu$ be a probability measure on $G$ supported inside $F$. There exist a sequence $(\sigma_i)_{i\geq 1}$ of elements in $B^{G}_{d'}(e,C)$ and subsets $E_j\subset F^j$ for every integer $j\geq 2$ such that the following assertions hold, for every integer $k\geq 2$:
	\begin{enumerate}
		\item for any $k$-tuple $(g_1,\dots,g_k)\in E_k$,
		\begin{equation*}
		\ell(g_1\sigma_1g_2\sigma_2\cdots g_{k-1}\sigma_{k-1}g_k)\geq \ell(g_1)+\cdots +\ell(g_k)-(k-1)c\;;
		\end{equation*} 
		\item $\nu^{k}(E_k)\geq |B^{G}_{d'}(e,C)|^{-(k-1)}$.
	\end{enumerate}
\end{lem}
\begin{proof}
	Define a map $F\times F\to B^{G}_{d'}(e,C)$ assigning to each pair $(x,y)\in F^{2}$ an element $\sigma_{xy}\in B^{G}_{d'}(e,C)$ such that $d(e,x\sigma_{xy}y)\geq d(e,x)+d(e,y)-c$, whose existence is ensured by Proposition~\ref{boundeddistortion}. Applying the union bound to the probability measure $\nu\times \nu$ on $F^{2}$, we can find an element $\sigma_1\in B^{G}_{d'}(e,C)$ and a subset $E_2\subset F^{2}$ such that $\nu\times \nu(E_2)\geq |B^{G}_{d'}(e,C)|^{-1}$ and $\sigma_{xy}=\sigma_1$ for every $(x,y)\in E_2$. Hence the lemma is shown for $k=2$.

	Consider now the set $\overline{E_2}=\{x\sigma_1 y:(x,y)\in E_2 \}\subset G$, and define a map $\overline{E_2}\times F\to B^{G}_{d'}(e,C)$ as before, via Proposition~\ref{boundeddistortion}. We endow $\overline{E_2}$ with the push-forward $\nu_2$ of the restriction of $\nu\times \nu$ to $E_{2}$, normalized to be a probability measure, under the map $\theta_2\colon E_2\ni (x,y)\mapsto x\sigma_1y \in\overline{E_2}$. Arguing as before, there exist an element $\sigma_2\in B^{G}_{d'}(e,C)$ and a subset $\widetilde{E_3}\subset \overline{E_2}\times F$ such that $\nu_2\times \nu(\widetilde{E_3})\geq |B^{G}_{d'}(e,C)|^{-1}$ and $\sigma_{xy}=\sigma_2$ for every $(x,y)\in \widetilde{E_3}$. Set $E_3\coloneqq (\theta\times \text{id}_F)^{-1}(\widetilde{E_3})$, which is a subset of $F^{3}$. If $(g_1,g_2,g_3)\in E_3$, then
	\begin{equation*}
	\ell(g_1\sigma_1g_2\sigma_2g_3)\geq \ell(g_1\sigma_1g_2)+\ell(g_3)-c\geq \ell(g_1)+\ell(g_2)-c+\ell(g_3)-c=\ell(g_1)+\ell(g_2)+\ell(g_3)-2c,
	\end{equation*}
	where the first inequality stems from the fact that $(g_1\sigma_1g_2,g_3)\in \widetilde{E_3}$, while the second one from $(g_1,g_2)\in E_2$. Furthermore, $\nu \times\nu \times \nu(E_3)=\bigl(\nu_2\times \nu(\widetilde{E_3})\bigr)\cdot\bigl( \nu\times \nu(E_2)\bigr)\geq |B^{G}_{d'}(e,C)|^{-2}$.
	
	A straightforward iteration of this procedure gives both assertions of the lemma for every $k$.\linebreak 
\end{proof}

We may now proceed with the proof of existence the weak LDP. The argument follows the proof of~\cite[Prop.~4.3]{Corso} almost verbatim: we leverage the criterion expressed in Proposition~\ref{criterion}. 

\begin{proof}[Proof of Theorem~\ref{main}: existence of weak LDP]
 We aim to establish the validity of the assumption in Proposition~\ref{criterion}. For the sake of reaching a contradiction, suppose that there exists $x\in \R_{\geq 0}$ such that (notice that $I(x)\geq -\limsup_n\frac{1}{n}\log{\mu_n(V)}$ for any open set $V$ containing $x$)
\begin{equation}
\label{contradiction}
I(x)>\supl_{x\in V\text{open}}-\ub(V)\;.
\end{equation}
Necessarily, $x$ is strictly positive; indeed, when $x=0$, subadditivity of $\ell$ implies that the limit $\lim_n\frac{1}{n}\log{\mu_n}(B(0,\eps))$ exists in $[-\infty,0]$ for every $\eps>0$.

Spelling out~\eqref{contradiction}, there exist $\delta,\eta>0$ such that 
\begin{equation}
\label{contr}
-\lb(B(x,\delta))>\biggl(\supl_{\rho>0}-\ub(B(x,\rho))\biggr) +\eta\;.
\end{equation}
Fix a positive real number $\rho$ such that $\rho<\inf\{x,\delta\}$, and select a strictly increasing sequence $(n_j)_{j\geq 1}$ of natural numbers for which
\begin{equation}
\label{secondcontr}
\lb(B(x,\delta))<\frac{1}{n_j}\log{\mu_{n_j}((B(x,\rho)))}\;\; -\eta\;.
\end{equation}
Lightening notation, set
\begin{equation}
\label{notation}
\alpha\coloneqq\lb(B(x,\delta)),\; \beta_j\coloneqq\frac{1}{n_j}\log{\mu_{n_j}((B(x,\rho)))} \text{ for every }j\geq 1.
\end{equation} 
We now intend to show that, as a matter of fact, the inequality $\alpha\geq \beta_{j}-\eta$ holds provided that $j$ is taken to be sufficiently large. In light of~\eqref{secondcontr}, the desired contradiction is then achieved.  

As the random walk is irreducible, we may choose, for each group element $\sigma \in B^{G}_{d'}(e,C)$, an integer $t(\sigma)\geq 1$ and and a real number $p(\sigma)\in (0,1]$ such that $\sigma$ is attained by the random walk within $t(\sigma)$ steps with probability $p(\sigma)$. Set $\overline{t}\coloneqq \sup\{t(\sigma):\sigma \in B^{G}_{d'}(e,C) \}$ and $\overline{p}=\inf\{p(\sigma):\sigma \in B^{G}_{d'}(e,C) \}$; notice that $\overline{t}<\infty$ and $\overline{p}>0$ as $B^{G}_{d'}(e,C)$ is a finite set.

Choose an integer $j_0\geq 1$ so that $n_{j_0}\geq (\overline{t}x+c)(\delta-\rho)^{-1}$,
and define $F\coloneqq\{g\in G:\ell(g)\in n_{j_0}B(x,\rho) \}$, so that  $e^{\beta_{j_0}n_{j_0}}=\P(Y_{n_{j_0}}\in F)$.
It is also convenient to fix a sequence $(\tilde{Y}_i)_{i\geq1}$ of independent copies of $Y_{n_{j_0}}$. We apply Lemma~\ref{selection} to the set $F$ and to the normalized restriction of the law of $Y_{n_{j_0}}$ to the set $F$. There exist thus a sequence $(\sigma_i)_{i\geq 1}$ of elements of $B^{G}_{d'}(e,C)$ and, for every integer $k\geq 2$, a subset $E_k$ of $F^{k}$ such that $\P((\tilde{Y}_1,\dots,\tilde{Y}_k)\in E_k)\geq |B^{G}_{d'}(e,C)|^{-(k-1)}e^{k\beta_{j_0}n_{j_0}}$ and, for any choice of elements $(g_1,\dots,g_k)\in E_k$, 
\begin{equation}
\label{weaklengthadd}
\ell(g_1\sigma_1g_2\sigma_2\cdots g_{k-1}\sigma_{k-1}g_k)\geq \ell(g_1)+\cdots\ell(g_k)-(k-1)c\;.
\end{equation}

Suppose now $(g_1,\dots,g_k)$ is a $k$-tuple in the set $E_k$. Then, combining~\eqref{weaklengthadd} with subadditivity of $\ell$, and setting $\overline{c}=\sup\{c,C \}$, we get 
\begin{equation*}
\ell(g_1)+\cdots +\ell(g_k)-k\overline{c}\leq \ell(g_1\sigma_1g_2\sigma_2\cdots g_{k_1}\sigma_{k-1}g_{k})\leq \ell(g_1)+\cdots \ell(g_k)+k\overline{c}\;,
\end{equation*}
which, together with the fact that $E_k$ is a subset of $F^k$, delivers
\begin{equation*}
k\bigl(n_{j_0}(x-\rho)-\overline{c}\bigr)<\ell(g_1\sigma_1g_2\sigma_2\cdots g_{k_1}\sigma_{k-1}g_{k})<k\bigl(n_{j_0}(x+\rho)+\overline{c}\bigr)\;.
\end{equation*}

The sequence $(n_k)_{k\geq 1}$ defined by $n_k\coloneqq kn_{j_0}+t(\sigma_1)+\cdots t(\sigma_{k-1})$ is non-lacunary, that is, it satisfies the assumption of Lemma~\ref{lowerbound}. Now choose $\rho'<\delta$ such that $n_{j_0}(\rho'-\rho)\geq \overline{t}x+c$; such a $\rho'$ exists by virtue of our choice of $j_0$, and ensures that the length of the element $g_1\sigma_1g_2\sigma_2\cdots g_{k-1}\sigma_{k-1}g_k$ belongs to $n_kB(x,\rho')$ whenever $(g_1,\dots,g_k)\in E_k$, for any $k\geq 2$.

We may now estimate, for every $k\geq 2$,
\begin{equation*}
\begin{split}
\P(\ell(Y_{n_k})\in n_kB(x,\rho'))&\geq \P\bigl((\tilde{Y}_1,\dots,\tilde{Y}_{k})\in E_k\bigr)\cdot p(\sigma_1)\cdots p(\sigma_{k-1})\\
&\geq e^{k\beta_{j_0}n_{j_0}}|B^{G}_{d'}(e,C)|^{-(k-1)}\;\overline{p}^{k-1}\\
&\geq e^{n_k(\beta_{j_0}-\eta)}\;,
\end{split}
\end{equation*}
where the last inequality is a consequence of our choice $n_{j_0}\geq \eta^{-1}\bigl(\log{|B^{G}_{d'}(e,C)|}-\log{\overline{p}}\bigr)$, whereas the first one derives from independence and stationarity of the process $(X_{n})_{n\geq 1}$.
Lemma~\ref{lowerbound} allows to deduce that $\alpha\geq \beta_{j_0}-\eta$, as desired. 

\end{proof}

We now turn to the proof that the rate function $I\colon \R_{\geq 0}\to [0,\infty]$, defined as in~\eqref{suplower} and governing the weak LDP satisfied by the sequence $\bigl(\frac{1}{n}\ell(Y_n)\bigr)_{n\geq 1}$, is convex. The argument is taken from~\cite[Prop.~5.1]{Corso}.

\begin{proof}[Proof of Theorem~\ref{main}: convexity of the rate function]
	Since $I$ is lower semicontinuous, it is sufficient to show (cf.~the proof of~\cite[Prop.~5.1]{Corso}) that it is mid-point convex, namely that
	\begin{equation*}
	I\biggl(\frac{1}{2}x_1+\frac{1}{2}x_2\biggr)\leq \frac{1}{2}I(x_1)+\frac{1}{2}I(x_2) \text{ for any real values }x_2>x_1\geq 0\;.
	\end{equation*} 
	Suppose this is not the case, for the purpose of a contradiction. Then, bearing in mind the equality in~\eqref{supupper} established in the first part of the proof of Theorem~\ref{main}, there exist $0\leq x_1<x_2$ and $\delta,\eta>0$ such that, for any real numbers $\rho_1,\rho_2>0$,
	\begin{equation}
	\label{wrongineq}
	\begin{split}
	\ub\biggl(&B\biggl(\frac{1}{2}x_1+\frac{1}{2}x_2,\delta\biggr)\biggr)
	<\\
	&<\frac{1}{2}\biggl(\lb(B(x_1,\rho_1))+\lb(B(x_2,\rho_2))\biggr)-\eta\;.
	\end{split}
	\end{equation}. 
	Fix $\rho\coloneqq \rho_1=\rho_2<\delta$. For a sufficiently large $n_0$ and every $n\geq n_0$, we claim that there exists $\phi(n)\in\{2n,\dots,2n+\overline{t} \}$ such that 
	\begin{equation}
	\label{rightineq}
	\frac{1}{\phi(n)}\log{\mu_{\phi(n)}\biggl(B\biggl(\frac{1}{2}x_1+\frac{1}{2}x_2,\delta\biggr)\biggr)}\geq \frac{1}{2}\biggl(\frac{1}{n}\log{\mu_n(B(x_1,\rho))}+\frac{1}{n}\log{\mu_n(B(x_2,\rho))}\biggr)-\eta\;.
	\end{equation}
	
	Letting $n$ vary over an arithmetic progression for which the corresponding sequence of $\phi(n)$ is strictly increasing, it is clear that we obtain a contradiction to~\eqref{wrongineq}.
	
	It remains to prove the claim just stated. Fix some integer $n_0\in \N$ with
	\begin{equation*}
	n_0\geq \frac{1}{2(\delta-\rho)}\sup\biggr\{C,c+\overline{t}\biggl(\frac{x_1+x_2}{2}\biggr) \biggr\}\;.
	\end{equation*}
	Let $n\geq n_0$, and define $F_i=\{g\in G:\ell(g)\in nB(x_i,\rho) \}$, $i=1,2$. Let also $\tilde{Y}_1,\tilde{Y}_2$ be two independent copies of $Y_{n}$. Adapting the proof of Lemma~\ref{selection} appropriately, we deduce that there is an element $\sigma\in B^{G}_{d'}(e,C)$ and a subset $E\subset F_1\times F_2$ such that
	\begin{equation*}
	 \P((\tilde{Y}_1,\tilde{Y}_2)\in E)\geq |B^{G}_{d'}(e,C)|^{-1}\P(Y_n\in F_1)\P(Y_{n}\in F_2)
	\end{equation*} 
	 and
	\begin{equation}
	\label{weaklengthadditivity}
	\ell(g_1\sigma g_2)\geq \ell(g_1)+\ell(g_2)-c \quad\text{for any }(g_1,g_2)\in E\;.
	\end{equation}
	In particular, \eqref{weaklengthadditivity} together with the choice of $n_0$ imply that
	\begin{equation*} \ell(g_1\sigma g_2)\in (2n+t(\sigma))B((x_1+x_2)/2,\delta)\quad \text{whenever  }(g_1,g_2)\in E\;;
	\end{equation*} \eqref{rightineq} is thus satisfied for $\phi(n)\coloneqq 2n+t(\sigma)$.
 
\end{proof}

\subsection*{Concluding remarks} It is currently unknown to the author whether Proposition~\ref{boundeddistortion} holds beyond the realm of relative word distances on relatively hyperbolic groups and of isometric actions with contracting elements (cf.~Remark~\ref{moregeneral}) . In the presence of such a generalization, the argument leading to Theorem~\ref{main} would carry over unaffectedly. On the other hand, it would be desirable to dispense with the irreducibility assumption on the random walk, replacing it with a milder non-degeneracy condition. In~\cite{Corso} the peculiar structure of free products allows for a considerable weakening of irreducibility; here, a further refinement of Proposition~\ref{boundeddistortion} seems to be needed.

\end{document}